\documentclass[10pt]{amsart}
\usepackage{amssymb,amsmath,enumerate,graphics,color}
\usepackage{comment}
\usepackage{ae,aecompl}
\usepackage{tikz}
\usepackage{amsthm}

\newtheorem{theorem}{Theorem}[section] 

\newtheorem{result}{Result}

\newcommand{\pseudoconcat}{\odot}

\newtheorem{proposition}[theorem]{Proposition}

\newtheorem{corollary}[theorem]{Corollary}
\newtheorem{lemma}[theorem]{Lemma}

\newtheorem*{mainresult}{Main Result}

\theoremstyle{definition}

\newcommand{\PP}{{\mathbb P}}

\newcommand{\seq}[1]{\mathbf{#1}}
\newcommand{\reverse}[1]{{{#1}^{*}}}
\newcommand{\V}{\ensuremath{\mathcal{L}}}
\newcommand{\x}{\mathbf{x}}
\newcommand{\concat}{\cdot}

\usepackage{lmodern}

\newcommand{\largo}[1]{\ell({#1})}

\begin{document}


\title[Proper caterpillars]{Proper caterpillars are distinguished by their symmetric chromatic function}


\author{Jos\'e Aliste-Prieto and Jos\'e Zamora}
\address{Jos\'e Aliste Prieto\\
	 Centro de Modelamiento Matematico\\
	 Universidad de Chile\\
	 Blanco Encalada 2120 7to. piso, Santiago\\
	 Chile}
\curraddr{Departamento de Matem\'aticas\\
	 Universidad Andres Bello\\
	 Republica 220, Santiago\\
	 Chile}
\email{jose.aliste@unab.cl}
\address{Jos\'e Zamora\\ Departamento de Matem\'aticas\\
        Universidad Andres Bello\\
        Republica 220, Santiago\\
	Chile}
\email{josezamora@unab.cl}

\thanks{J.A.-P. is supported by Fondecyt Postdoctoral grant 3100097 of CONICYT Chile.}

\begin{abstract}
This paper deals with the so-called Stanley conjecture, which asks whether they are non-isomorphic trees with the same symmetric function generalization of the chromatic polynomial. 
By establishing a correspondence between caterpillars trees and integer compositions, we prove that caterpillars in a large class (we call trees in this class proper) have the same symmetric chromatic function generalization of the chromatic polynomial if and only if they are isomorphic. 
\end{abstract}
\maketitle
\section{Introduction}

The weighted graph polynomial $U_G$ \cite{Noble99} and the symmetric chromatic function $X_G$ \cite{Stanley95} of a graph $G$ are powerful invariants. They have been actively studied and have diverse applications as they encode much of the combinatorics of the given graph. In particular, many well-known isomorphism invariants such as the Tutte polynomial and the chromatic polynomial can be obtained as evaluations of them. A natural question about either $X_G$ or $U_G$ is to decide whether they are complete isomorphism invariants. More precisely, Do there exist non-isomorphic graphs with the same symmetric chromatic function (resp. the same weighted graph polynomial)? The answer to both questions is affirmative: Examples of non-isomorphic graphs with the same symmetric
chromatic function can be found in \cite{Stanley95}; on the other hand, one can find non-isomorphic graphs with the same weighted graph polynomial combining the work in \cite{Sarmiento00,Bry81}. However, these questions remain open when restricted to trees. In fact, they are equivalent, due to the fact that $X_G$ and $U_G$ can be recovered one from each other when the graph $G$ is a tree (see \cite[Theorem 6.1]{Noble99}). So the question stands: \emph{Do there exist non-isomorphic trees with the same symmetric chromatic function?}. This question is often referred to as \emph{Stanley's question} or \emph{Stanley conjecture} \cite{Stanley95}. 

Despite of the importance of the symmetric function generalization of the chromatic polynomial, not much is known in the literature about this question. We now review some known partial results towards a solution that appeared in \cite{Martin08}. First, we need to recall some definitions. Given a class of trees, we say that $X_G$ distinguishes among this class if trees in the class with the same symmetric chromatic function must be isomorphic.  A \emph{caterpillar} is a tree
where all the internal edges form a path, which is referred to as the \emph{spine} of the caterpillar. A caterpillar is \emph{proper} if 
each vertex in the spine is adjacent to a least one leaf. The spine  induces a linear structure, which 
allows us to define the \emph{leaf-sequence} of a caterpillar: To each vertex in the spine, we associate the number of leaves adjacent to it. In \cite{Martin08}, it is shown that $X_G$ distinguishes among caterpillars with a palindromic leaf-sequence and among proper caterpillars having a leaf-sequence with all its components being distinct. Thus, determining whether $X_G$ distinguishes among all caterpillars seems to be a natural step towards the solution of Stanley's question. In this paper, we obtain the following:

\begin{mainresult}
The symmetric chromatic function distinguishes among proper caterpillars. 
\end{mainresult}

To this purpose, we give a sufficient condition for caterpillars to have a distinct symmetric function generalization of the chromatic polynomial. Next, we describe a natural embedding of proper caterpillars into the set of integer compositions. We introduce a polynomial for compositions, which we call the $\V$-polynomial, that mimics the weighted graph polynomial of Noble and Welsh. We also show that the $\V$-polynomial of an integer composition can be computed as an evaluation of the weighted graph polynomial of the corresponding proper caterpillar. 
Finally, we observe that $\V$-polynomial is equivalent to the multiset of partition coarsenings defined by Billera, Thomas and van Willigenburg in \cite{Billera}, and then we combine  their characterization of integer compositions having the same multiset of partition coarsenings with our sufficient condition to establish our main result. 

This paper is organized as follows: In Section 2, we study the correspondence between caterpillars and integer compositions. In Section 3.1, we review the results of \cite{Billera} needed for the proof of our main result, while in Section 3.2 we establish our sufficient condition, and finally in Section 3.3 we combine these and give the proof of our main result. 

\section{Caterpillars versus compositions}

\subsection{Compositions and the $\V$-polynomial}
Let $\PP$ denote the set of positive integers. Let $n$ be a positive integer. A composition $\beta$ of $n$, denoted $\beta\models n$, is a list $\beta_1\beta_2\ldots\beta_k$ of positive integers such that $\sum_i\beta_i = n$. We refer to each of the $\beta_i$ as components, and say that $\beta$ has \emph{length} $\largo{\beta}=k$ and \emph{size} $|\beta|=n$. The set of all compositions of $n$ will be denoted by $\mathcal{C}_n$. The set of all compositions is given by 
\[\mathcal{C}=\bigcup_{n\in\PP}\mathcal{C}_n,\]
and is equal to the set of all non-empty words with alphabet $\PP$. The \emph{reverse} of a composition $\beta=\beta_1\beta_2\ldots\beta_k$ is the composition  $\reverse{\beta} = \beta_k\ldots \beta_2\beta_1$. A composition $\beta$ is a \emph{palindrome} if and only if $\reverse{\beta} = \beta$. We say that  $\alpha\sim_* \beta$  if either $\alpha=\beta$ or $\alpha=\reverse{\beta}$ and denote by $[\beta]^*:=\{\beta,\reverse{\beta}\}$ the corresponding \emph{reverse-class}.

Given two compositions  $\alpha=\alpha_1\alpha_2\ldots\alpha_k$ and  $\beta=\beta_1\beta_2,\ldots\beta_l$, recall that the \emph{concatenation}  is given by  
\[\alpha_1\alpha_2\cdots\alpha_k\concat \beta_1\beta_2\cdots\beta_l = \alpha_1\alpha_2\cdots\alpha_k\beta_1\beta_2\cdots\beta_l,\]
and that the \emph{near-concatenation} of $\alpha$ and $\beta$ is given by 
\[\alpha\pseudoconcat\beta:= \alpha_1 \alpha_2 \ldots \alpha_{k-1} (\alpha_{k}+\beta_1) \beta_2 \beta_3\ldots \beta_l.\]

Next, recall the following partial order on compositions. Given two compositions $\alpha$ and $\beta$ in $\mathcal{C}$,  we say that $\beta$ is a \emph{coarsening} of $\alpha$, denoted $\beta\succeq\alpha$ if $\beta$ can be obtained from $\alpha$ by adding consecutive components from $\alpha$. That is to say, there exists  an increasing finite sequence $1=j_0<j_1 < j_2 < ... < j_{i} < j_{i+1}=\largo{\alpha}+1$ of  integer indices such that 
\[\beta = \alpha_{j_0}\cdots\alpha_{j_{1}-1}\pseudoconcat\alpha_{j_1}\cdots\alpha_{j_{2}-1}\pseudoconcat \ldots\pseudoconcat \alpha_{j_i}\cdots\alpha_{j_{i+1}-1}.\]
For convenience, we will denote $\beta_{i,k}=\beta_i\beta_{i+1}\cdots\beta_{k}$ to shorten the above notation.

Observe that by a well-known result of McMahon \cite{macmahon},
 $(\mathcal{C}_n,\preceq)$ is isomorphic as a poset to the Boolean poset of dimension $n-1$, \emph{i.e.}, the set of all subsets of $\{1,2,\ldots,n-1\}$ ordered by inclusion.

A \emph{partition} of $n$  is a composition $\lambda$ of $n$ where the components satisfy $\lambda_1\geq\lambda_2\geq\ldots\geq \lambda_l$. The \emph{type} of a composition  $\beta$, denoted by $\lambda(\beta)$, is the partition obtained by reordering the components of $\beta$ in a weakly decreasing way.

Let $\x=x_1,x_2,\ldots$ be an infinite collection of commuting indeterminates. Given a partition $\lambda=\lambda_1\lambda_2\cdots\lambda_l$ of $n$, define $\x_\lambda:=x_{\lambda_1}x_{\lambda_2}\cdots x_{\lambda_l}$.  The \emph{composition-lattice} polynomial of a composition $\beta$ is defined by 
\[\V(\beta,\x) = \sum_{\alpha\succeq \beta} \x_{\lambda(\alpha)}.\] 

If $P$ is any polynomial in $\x$, and $\lambda$ is a partition, then $[\x_{\lambda}]P$ will denote the coefficient of $\x_{\lambda}$ when $P$ is expanded in the standard monomial basis.

\subsection{The weighted graph polynomial}
\label{sec.applications}

Let $G=(V,E)$ be a simple graph. If $A\subseteq E$, then   $G|_A$  is the graph obtained from $G$ after deleting all the edges in the complement of $A$ from $G$ (but keeping all the vertices). We recall the definition of the weighted graph polynomial (a.k.a. the $U$-polynomial), originally introduced by Noble and Welsh \cite{Noble99}. Note that we give  the definition only for simple graphs (it is possible to  define the $U$-polynomial  for graphs with loops and parallel edges but we will not need this generality here).

 The rank of $A$, denoted $r(A)$, is given by 
\[
	r(A) = |V| - k(G|_A),  
\] 
where $k(G|_A)$ denotes the number of connected components of $G|_A$. Let
$\lambda(A) = \lambda_1\lambda_2\cdots\lambda_k$ be the partition of $|V|$ induced by the connected
components of $G|_A$, that is, $\lambda_1,\lambda_2,\ldots,\lambda_k$ are the cardinalities of 
the connected components of $G|_A$. The \emph{$U$-polynomial} of $G$ is defined by:
\[U_G(\x,y)=\sum_{A \subseteq E} \x_{\lambda(A)}(y-1)^{|A| - r(A)} .\]

A graph  $G$ is \emph{$U$-unique} if the $U$-polynomial of every graph that is not isomorphic to $G$ is different from the  $U$-polynomial of $G$.

When $G$ is a tree $T$, it is easy to check that $r(A) = |A|$. Thus, the $U$-polynomial of $T$ reads 
\[U_T(\x)=\sum_{A \subseteq E} \x_{\lambda(A)}.\]
This implies, in particular, that a tree $T$ is $U$-unique if and only if the $U$-polynomial of every \emph{tree} that is not isomorphic to $T$ is different from the $U$-polynomial of $T$.

Alternatively, by associating monomials, we get
\[ U_T(\x)=\sum_{\lambda \vdash |V|} c_\lambda(T) \x_\lambda,\]
where $c_\lambda(T)$ denotes the 
the number of subsets $A\subseteq E$ such that $\lambda(A) = \lambda$, and the sum is over all the partitions of $|V|$.

\subsection{Caterpillars and the $U^L$-polynomial}

Recall that a  tree $T$ is a \emph{caterpillar} if the induced subgraph on the internal
vertices is a non-trivial path $P(T)$, which is called the \emph{spine} of $T$. 
As it is usual, we will identify $P(T)$ with its set of edges and let $L(T)=E\setminus P(T)$ be the set of leaves-edges of $T$.  
A caterpillar $T$ is \emph{proper} if every internal vertex of $T$ is adjacent to at least one leaf. 
 
The \emph{restricted weighted polynomial}, or $U^L$-polynomial, of a caterpillar $T$ 
is defined by 
\[ U^L_T(\x) = \sum_{A \subseteq E(T), L(T)\subseteq A} \x_{\lambda({A})}.\]

\begin{proposition}\label{l:ugorrofromu}
For every caterpillar $T$, we have
\begin{equation}
\label{eq:UevalUL}
U_T(x_1=0,x_2,x_3,\ldots) = U^L_T(x_1 = 0,x_2,x_3,\ldots).
\end{equation}
Furthermore, if $T$ is proper, then $U^L$ does not depend on $x_1$. In particular, in such case $U_T^L$
is an evaluation of the $U$-polynomial of $T$.
\end{proposition}
\begin{proof}
Let $A\subseteq E$ be such that $L(T)$ is not a subset of $A$, and pick an edge $e\in L(T)$ that is also in the complement of $A$. Observe that the leaf adjacent to $e$ is an isolated vertex in $G|_A$. This implies that $1$ is a part of $\lambda(A)$, which means that  $x_1$ divides $\x_{\lambda(A)}$. It follows that $\x_{\lambda(A)}|_{x_1=0} = 0$. Thus,   
\[U_T(x_1=0,x_2,x_3,\ldots) = \sum_{A\subseteq E, L(T)\subseteq A} x(\lambda_A)|_{x_1 = 0} = U_T^L(x_1=0,x_2,x_3,\ldots),\]
which establishes \eqref{eq:UevalUL}. To get the conclusion, observe that if $T$ is proper and $A\subseteq E(T)$ contains $L(T)$, then $T|_A$ does not have isolated vertices, which effectively means that $U_T^L$ does not depend on $x_1$. Hence, the last assertion follows from \eqref{eq:UevalUL}.

\end{proof}

\subsection{Caterpillars versus compositions}
Let $\mathcal{T}_+$ be the family of all proper caterpillars and $\mathcal{P}$ be the set of reverse-classes of all integer compositions. There is a natural embedding of $\mathcal{T}_+$ into $\mathcal{P}$. Indeed, suppose that the internal vertices of $T$ are enumerated as $\{v_1,v_2,\ldots v_k\}$, where 
$v_i$ is adjacent to $v_{i+1}$ for each $i\in\{1,\ldots,k-1\}$. In other words, the spine of $T$ is the path $P(T)=v_1 v_2\ldots v_k$. Then, for each $i\in\{1,\ldots,k\}$, define $\beta_i$ to be  the number of vertices in the connected component of $T|_{L(T)}$ that contains $v_i$. Finally set 
\[\Phi(T) = [\beta_1 \beta_2\cdots \beta_k]^*.\] 

\begin{lemma}
\label{lema.VimplicaU}
The map $\Phi:\mathcal{T}_+\rightarrow\mathcal{P}$ is one-to-one.
\end{lemma}
\begin{proof}
Let $\mathcal{P}_+$ be the image of $\mathcal{T}_+$ by $\Phi$. We construct $\Psi:\mathcal{P}_+\rightarrow\mathcal{T}_+$, which is the inverse of $\Phi$, explicitly. Given $[\beta]^*$ in $\mathcal{P}^+$, let $\tilde{T}$ be a path with $\largo{\beta}$ vertices, that is, $\tilde{T}=v_1\ldots v_{\ell(\beta)}$. Next, for every $i\in\{1,\ldots,\largo{\beta}\}$, we attach $\beta_i - 1$ leaves to the vertex $v_i$ and denote by $T=\Psi(\beta)$ the caterpillar generated by  this process. It is clear that $T$ does not depend on the choice of $\beta$ in the reverse-class. Moreover, since $[\beta]^*$ belongs to $\mathcal{P}^+$, it is clear that $\beta_i>1$ for all $i$, which means that $T$ is proper. Hence, $\Psi$ is well-defined. Finally, it is direct to check that $\Psi$ is the inverse of $\Phi$. 
\end{proof}

Observe that since the $\V$-polynomials of a composition and its reverse coincide, we can define 
\[\V(\Phi(T),\x) = \V(\beta,\x),\qquad \beta\in\Phi(T).\]
\begin{proposition}
\label{ULequalsV}
For every $T\in\mathcal{T}_+$ we have 
\[U_L^T(\x) = \V(\Phi(T),\x).\]
\end{proposition}
\begin{proof}
Fix $\beta\in\Phi(T)$ and an orientation of the spine $P(T)=v_1\ldots v_n$ such that sequence of the number of vertices of the connected components of $T|_{L(T)}$ 
coincides with $\beta$. We will establish a correspondence between compositions $\alpha\succeq \beta$ and sets $A\subseteq E(T)$ containing $L(T)$ that satisfy the relation $\lambda(A)=\lambda(\beta)$. Indeed, suppose that $A\subseteq E(T)$ contains $L(T)$. Then the edges in $E\setminus A$ are all internal, which means that $E\setminus A = \{v_{j_1}v_{j_1+1},v_{j_2}v_{j_2+1},\ldots,v_{j_k}v_{j_k+1}\}$ with  $j_1<j_2<\ldots < j_k$. By defining 
\[\alpha(A)=|\beta_{1}\cdots\beta_{j_1}||\beta_{j_1+1}\cdots\beta_{j_2}|\ldots|\beta_{j_{k-1}+1}\cdots\beta_{j_k}||\beta_{j_k+1}\cdots\beta_{n}|,\]
it is clear that $\alpha(A)\succeq \beta$ and  $\lambda(\alpha(A)) = \lambda(A)$. Conversely, if $\alpha\succeq \beta$, then by definition, there exist $1=j_0<j_1<j_2<\ldots<j_i<j_{i+1}=\largo{\beta}+1$ such that 
\[\alpha=\beta_{j_0}\cdots\beta_{j_1-1}\pseudoconcat\beta_{j_1}\cdots\beta_{j_2-1}\pseudoconcat\ldots\pseudoconcat\beta_{j_{i-1}}\cdots\beta_{j_i-1}\pseudoconcat\beta_{j_i}\cdots\beta_{j_{i+1}-1}.\]
By defining 
\[A(\alpha) = L(T) \cup \{v_{j_1-1}v_{j_1},v_{j_2-1}v_{j_2},\ldots,v_{j_i-1}v_{j_i}\},\]
we check that $\lambda(A(\alpha))=\lambda(\alpha)$. It is left to the reader to check that
$A(\alpha(A)) = A$ and $\alpha(A(\alpha))=\alpha$. Finally, by using the last correspondence, we get
\[ U^L_T(\x) = \sum_{A \subseteq E(T), L(T)\subseteq A} \x_{\lambda({A})} = \sum_{A \subseteq E(T), L(T)\subseteq A} \x_{\lambda({\seq{v(A)}})} = \sum_{\alpha\succeq \beta} \x_{\lambda(\alpha)}
=\V(\beta,\x).\]
\end{proof}
The next corollary follows direct from Proposition \ref{ULequalsV} and Proposition \ref{l:ugorrofromu}.
\begin{corollary}
\label{UimpliesL}
Let $T$ and $T'$ be two proper caterpillars with the same $U$-polynomial. Then, $\Phi(T)$ and $\Phi(T)'$ have the same $\V$-polynomial.
\end{corollary}
We say that  $\alpha\sim_\V \beta$  if $\V(\alpha,\x)=\V(\beta,\x)$ and denote by $[\beta]_\V=\{\alpha\in\mathcal{C}\mid \beta\sim_\V \alpha\}$ the corresponding \emph{$\V$-class}. A composition $\beta$ is \emph{$\V$-unique} if and only if $[\beta]_\V = [\beta]_*$. 
\begin{corollary}
\label{thm.V-unico}
Let $T$ be a proper caterpillar and $\beta\in\Phi(T)$. Suppose that $\beta$ is $\V$-unique. Then, $T$ is $U$-unique.
\end{corollary}
\begin{proof}
Suppose $T'$ is a tree such that $U(T,\x) = U(T',\x)$. In  \cite{Martin08}, it is proved that whether a tree is a caterpillar or not can be recognized from $U$. Thus, since $T$ and $T'$ have the same $U$-polynomial, 
and from $[\x_{\lambda(\beta)}]U_T = 1$, we can recognize that $T$' must also be a proper caterpillar. Let $\alpha\in\Phi(T')$. It follows from Proposition \ref{l:ugorrofromu} and Lemma \ref{lema.VimplicaU} that $\V(\alpha,\x)=\V(\beta,\x)$. The $\V$-uniqueness
of $\beta$ now implies that $\alpha\sim^* \beta$. Thus, $T'$ is isomorphic to $T$ and $T$ is $U$-unique.
\end{proof}

\section{Proof of the main result}
\subsection{Description of compositions with the same $\V$-polynomial}
\label{sec.def}


It is easy to see that $\V$-polynomial of a composition $\beta$ is equivalent to the multiset of partitions coarsenings of $\beta$
\[\mathcal{M}(\beta)=\{ \lambda(\alpha)\mid \alpha\succeq \beta\},\]
introduced in \cite{Billera}. Since the class of compositions that have the same multiset of partition coarsenings has been completely described in \cite{Billera}, we get a complete description of the $\V$-class of a given composition. We recall now this description. When possible, we follow the notation in \cite{Billera}.

For convenience we write 
 \[\alpha^{\odot i}:=\underbrace{\alpha\pseudoconcat\alpha\pseudoconcat \ldots\pseudoconcat\alpha}_{i\text{ times}},\]
where $\alpha$ is a composition and $i$ is a positive integer. Given $\alpha\models n$ and $\beta\models m$, the composition 
$\beta\circ \alpha$ is defined by 
\[  \beta\circ \alpha = \alpha^{\odot {\beta_1}} \concat \alpha^{\odot{\beta_2}} \cdots  \alpha^{\odot{\beta_k}},\]
where $\largo{\beta}=k$. 
It is clear that $\beta\circ\alpha\models n m$. If a composition $\alpha$ is written  in the form $\alpha_1\circ\alpha_2\circ\cdots\circ\alpha_k$, then we call this a \emph{factorization} of $\alpha$. A factorization of $\alpha=\beta\circ\gamma$ is called \emph{trivial}
if any of the following conditions are satisfied:
\begin{enumerate}
\item one of the $\beta$,$\gamma$ is the composition $1$,
\item the compositions $\beta$ and $\gamma$ both have length $1$, 
\item the compositions $\beta$ and $\gamma$ both have all components equal to $1$. 
\end{enumerate}
A factorization $\alpha=\alpha_1\circ\alpha_2\circ\cdots\circ\alpha_k$ is \emph{irreducible} if no $\alpha_i\circ\alpha_{i+1}$ is a trivial factorization, and each $\alpha_i$ admits only trivial factorizations. In this case, each $\alpha_i$ is called an \emph{irreducible factor}. 
\begin{theorem}[{\cite[Theorem 3.6]{Billera}}]
\label{star_basic_prop} 
Each composition admits a unique irreducible factorization.
\end{theorem}

Let $\alpha\models n$ and $\alpha=\alpha_1\circ\alpha_2\circ\cdots\circ\alpha_k$ be the unique irreducible factorization of $\alpha$. Let $id$ denote the identity map in $\mathcal{C}$ and $R$ denote the reverse map, that is, $R(\alpha)=\reverse{\alpha}$ for all $\alpha\in\mathcal{C}$. The \emph{symmetry-class} of 
$\beta$ is defined by   
\[\operatorname{Sym}(\alpha):=\big\{T_1(\alpha_1)\circ T_2(\alpha_2)\circ\cdots\circ T_k(\alpha_k)\mid T_i\in\{id,R\} \text{ for all $i\in\{1,\ldots,l\}$}\big\}.\]  

\begin{theorem}[{\cite[Corollary 4.2]{Billera}}]
\label{prop.SymClass}
For every composition $\alpha$ we have 
\[[\alpha]_\V=\operatorname{Sym}(\alpha).\] 
\end{theorem}

\subsection{Caterpillars with distinct $U$-polynomials}

In this section, we give a very general sufficient condition for two proper caterpillars to have distinct $U$-polynomials. First, we need some notation.

Recall that $\alpha$ is lexicographically less than $\beta$, denoted $\alpha<_L\beta$, if one of the two conditions hold:
\begin{enumerate}
\item $\largo{\alpha}<\largo{\beta}$ and $\alpha_i=\beta_i$ for all $i\in\{1,\ldots,\largo{\alpha}$,
\item There exists $k\in\{1,\ldots,\largo{\alpha}$ such that $\alpha_k < \beta_k$ and $\alpha_i=\beta_i$ for all $i\in\{1,\ldots,k-1\}$.
\end{enumerate} 
Given two compositions $\alpha\neq \beta$ of the same length, let 
\begin{equation}
\label{def:k}
k(\alpha,\beta):=\min\{1\leq k\leq  \largo{\alpha} \mid \alpha_k\neq \beta_k\}
\end{equation}
denote the index where the first difference (from left to right) between $\alpha$ and $\beta$ appears. A composition $\beta$ is a \emph{prefix} of another composition $\gamma$ if there exists a composition $\alpha$ such that $\gamma = \beta\cdot\alpha$. Acordingly, $\beta$ is a suffix of $\gamma$ if there exists $\alpha$ such that $\gamma=\alpha\cdot\beta$.
\begin{theorem}
\label{t.condsufic}
Suppose $S$ and $T$ are two proper caterpillars such that $\Phi(S)=[\alpha\circ\gamma]^*$ and $\Phi(T)=[\beta\circ\gamma]^*$, where $\alpha,\beta$ and $\gamma$ belong to $\mathcal{C}$, and $\alpha$ and $\beta$ have the same size. If $\gamma$ is not a palindrome and $\alpha\neq\beta$, then the $U$-polynomials of $S$ and $T$ are distinct.
\end{theorem}
\begin{proof}
To avoid confusions, we sometimes write $(\alpha_1,\alpha_2,\ldots,\alpha_k)$ instead of $\alpha_1\alpha_2\ldots\alpha_k$
for a composition $\alpha$. W.l.o.g we may assume that $\gamma<_L\reverse{\gamma}$ and $\alpha<_L\beta$. Fix $\sigma=\alpha\circ\gamma$, $\tau=\beta\circ\gamma$ and $n=|\alpha|=|\beta|$. 
Fix also $a = |\alpha_{1,k(\alpha,\beta)}|$ and $b=|\gamma_{1,k(\gamma,\reverse{\gamma})}|$, where $k:=k(\alpha,\beta)$ and $k(\gamma,\reverse{\gamma})$ are defined by \eqref{def:k}. Since $\gamma<_L\reverse{\gamma}$, it is easy to  check that $[\x_{(b,|\gamma|-b)}]\V(\gamma) = 1$. Now consider $\delta=\delta_1\delta_2$, where 
\[\delta_1=a|\gamma|+b,\quad \delta_2= n|\gamma| - \delta_1.\]
We show that 
\[[\x_{\lambda(\delta_1,\delta_2)}]\V(\sigma)= 1.\]
Indeed, it is easy to check that $\delta\succeq \sigma$ and $\delta\succeq \tau$ since
the composition 
$\rho_1 =(\alpha_{1,k}\circ\gamma)\concat \gamma_{1,k(\gamma,\reverse{\gamma})}$ is a prefix for $\sigma$ while  $\rho_2 =(\alpha_{1,k}\circ\gamma)\pseudoconcat \gamma_{1,k(\gamma,\reverse{\gamma})}$ is a prefix for $\tau$, and both compositions have size $\delta_1$. Let us now suppose that 
$[\x_{\lambda(\delta_1,\delta_2)}]\V(\sigma)=2$, that is, there is a suffix $\phi$ of $\sigma$ such that $|\phi|=\delta_1$. By the definition of $\sigma$, this would imply the existence of a suffix $\psi$ of $\gamma$ such that $|\psi|=b$, which in turn would yield that $[x_{(b,|\gamma|-b)}]\V(\gamma)=2$, which is a contradiction. Hence, $[\x_{\lambda(\delta_1,\delta_2)}]\V(\sigma)=1$. 

Now we compute  $[\x_\lambda]U_S = \sharp\{ A\subseteq E(S)\mid \lambda(A)=\lambda\}$ for $\lambda=\lambda(1,\delta_1-1,\delta_2)$. 
Indeed, fix $A\subseteq E(S)$ such that $\lambda(A)=\lambda$. Since  $S$ is proper, it follows that $E(S)\setminus A=\{e_1,e_2\}$, where $e_1=v_iv_{i+1}$ is an internal edge and $e_2$ is a leaf.
 This means that $A'=A\cup\{e_2\}$ contains $L(S)$ and either $\lambda(A')$ equals $\{\lambda(\delta_1,\delta_2)\}$ or $\{\lambda(\delta_1-1,\delta_2+1)\}$. Since $A'$ contains all leaves, it corresponds to the sequence $\zeta:=|\sigma_{1,i}||\sigma_{i+1,\largo{\sigma}}|$ and we have $\lambda(\zeta)=\lambda(A')$. Let us see that necessarily $\zeta=\delta$. Indeed, supposing $\zeta=\reverse{\delta}$ implies that $(n-b,b)\succeq \gamma$ which is not possible. Next, supposing $\zeta=(\delta_1-1,\delta_2+1)$ implies, as we already know that $(\delta_1,\delta_2)\succeq\sigma$, that $(\delta_1-1,1,\delta_2)\succeq \sigma$, which is not possible since $S$ is proper. Finally, supposing $(\delta_2+1,\delta_1-1)\succeq\sigma$ implies that $(n-b+1,b-1)\succeq\gamma$. But this is not possible, since by hypothesis, if we let $l=k(\gamma,\reverse{\gamma})$, then we have $|(\reverse{\gamma})_{1,l}|>|\gamma_{1,l}| = b >|(\reverse{\gamma})_{1,l-1}|$ and 
 $|(\reverse{\gamma})_{1,l-1}|= |{\gamma}_{1,l-1}|= b - \gamma_{l} < b - 1$, where the last inequality follows from the fact that $S$ is proper.  Hence we have $\zeta=\delta$. Now, this means that $A'$ is indeed uniquely determined. This implies, in particular, that  
\[[\x_\lambda]U_S = \sharp L(\Psi(\rho_1)).\]

Using a similar argument, it is possible to show that 
\[[\x_\lambda]U_T = \sharp L(\Psi(\rho_2)).\]  

Thus, to finish the proof, it suffices to compute the number of leaves in $\Psi(\rho_1)$ and $\Psi(\rho_2)$ and show they are different. The following lemma allows us to compute the number of leaves in a proper caterpillar. To simplify notation, we say that a composition $\gamma\in\mathcal{C}$ is \emph{proper} if every element in $\gamma$ is larger than one.
 
\begin{lemma}
\label{lemma:leaves}
Suppose $\gamma\in\mathcal{C}$ is proper. Then, 
\[\sharp L(\Psi(\gamma))=|\gamma|-\largo{\gamma}.\]
\end{lemma}
\begin{proof}
Since $\gamma$ is proper, it is easy to see that $\Psi(\gamma)$ is also proper. Moreover,
if $v_1\ldots v_{\largo{\gamma}}$ denotes the spine of $\Psi(\gamma)$, it is direct from the definition of $\Psi$ that the number of leaves incident to $v_i$ is $\gamma_i-1$ for every $i\in\{1,\ldots,\largo{\gamma}\}$. This implies that $\sharp L(\Psi(\gamma))= \sum_i (\gamma_i - 1) =|\gamma|-\largo{\gamma}$, which is the desired conclusion.
\end{proof}
Motivated by Lemma \ref{lemma:leaves}, given $\gamma\in\mathcal{C}$, define $N(\gamma)=|\gamma|-\largo{\gamma}$. The following lemma resumes the properties of $N$.
\begin{lemma}
\label{hojasxy}
Suppose $\gamma$ and $\alpha$ are two proper compositions. Then, the following assertions hold:
\begin{enumerate}[(i)]
\item $N(\gamma\concat\alpha) = N(\gamma) + N(\alpha)$;
\item $N(\gamma\pseudoconcat\alpha) = N(\gamma)+ N(\alpha)+1$;
\item $N(\alpha\circ\gamma)) = N(\gamma)|\alpha| + N(\alpha)$.
\end{enumerate}
\end{lemma}
\begin{proof}
\emph{(i)} and \emph{(ii)} are clear from the definition of $N$. To show \emph{(iii)}, it follows from  \emph{(ii)}  that $N(\gamma^{\odot \alpha_i}) = \alpha_i N(\gamma) + \alpha_i-1$ for every $i\in\{1,\ldots \largo{\alpha}\}$. Since $\alpha\circ\gamma=\gamma^{\odot \alpha_1} \concat \gamma^{\odot  \alpha_2} \ldots \gamma^{\odot \alpha_{\largo{\alpha}}}$ by definition, it follows from \emph{(i)} that 
  \[N(\alpha\circ\gamma) = \sum_{i=1}^{\largo{\alpha}}(\alpha_i N(\gamma) + \alpha_i - 1) = N(\gamma)|\alpha|  +|\alpha| - \largo{\alpha}  = N(\gamma)|\alpha|  +N(\alpha).\]
\end{proof}

Now to finish the proof of the theorem, applying Lemma \ref{hojasxy}
it is easy to check that
\[
N (\rho_2) = N (\rho_1)+1.
\]
Since the sequences $\rho_2$ and $\rho_1$ are proper, it follows from Lemma \ref{lemma:leaves} that $\sharp L(\Psi(\rho_2))$ and $\sharp L(\Psi(\rho_1))$ are distinct. This implies that $[x_\lambda]U_T$ and $[x_\lambda]U_S$ are different and the conclusion now follows.
\end{proof}
\subsection{Proof of main result}
Now we are almost in position to give the proof of our main result. The last result we need is the following:
\begin{proposition}
\label{palindromes}
Suppose that $\beta$ is a palindrome. Then, $\beta$ is $\V$-unique.
\end{proposition}
\begin{proof}
By Theorem \ref{star_basic_prop}, and the fact that the reverse operation commutes with the product $\circ$, it is easy to check that $\beta$ is a palindrome if and only if all their irreducible factors are palindromes. It is easy to see that the symmetry class of an irreducible factor that is also a palindrome is equal to its reverse-class, which is a singleton.
Thus, by Theorem \ref{prop.SymClass}, the $\V$-class of $\beta$ is equal to $\{\beta\}$, which means that $\beta$ is $\V$-unique. 
\end{proof}

\begin{proof}[Proof of Main Result]
Suppose that $T$ and $T'$ are two proper caterpillars with the same $U$-polynomial. We assume 
by contradiction that $T$ and $T'$ are not isomoprhic. 
By Corollary \ref{l:ugorrofromu}, it follows that $\alpha\in\Phi(T)$ and $\beta\in\Phi(T')$ have the same $\V$-polynomial and $\alpha\not\sim^*\beta$. By Theorem \ref{prop.SymClass}, 
we have that $\operatorname{Sym}(\alpha)=\operatorname{Sym}(\beta)$. That is to say, if $\alpha=\alpha_1\circ\alpha_2\circ\cdots\circ\alpha_k$ and 
$\beta=\beta_1\circ\beta_2\circ\cdots\circ\beta_k$ are the irreducible factorizations of $\alpha$ and $\beta$, then 
\[\alpha_i\sim_*\beta_i\quad\text{for all $i\in\{1,\ldots,k\}$}.\]
On the other hand, by Proposition \ref{palindromes}, neither $\alpha$ nor $\beta$ can be palindromes. Hence, there exists $l$ such that $\alpha_l$ is not a palindrome, and $\alpha_i$ is a palindrome for every $i>l$. Moreover, since the product $\circ$ commutes with the reverse operations, we may assume that $\alpha_l = \beta_l$. By setting $\gamma=\alpha_l\circ\alpha_{l+1}\circ\cdots\circ\alpha_k$, it follows that 
\[\alpha=\alpha_1\circ\alpha_2\circ\cdots\circ\alpha_{l-1}\circ\gamma \quad \text{and}\quad \beta=\beta_1\circ\beta_2\circ\cdots\circ\beta_{l-1}\circ\gamma.\]
From this, it is direct to check that
\[\alpha_1\circ\alpha_2\circ\cdots\circ\alpha_{l-1}\neq\beta_1\circ\beta_2\circ\cdots\circ\beta_{l-1}. \]
Hence, by Theorem \ref{t.condsufic}, the $U$-polynomials of $T$ and $T'$
are distinct. This gives a contradiction, and thus concludes the proof. 
\end{proof}

\subsection*{Acknowledgments}
Both authors thank P. McNamara for pointing them to the work of L.Billera, H. Thomas and S. van Willigenburg soon after a previous version of this paper
was submitted to the arXiv. They also thank O. Carton,  E. Friedman, M. Matamala, R. Menares and M. L\"oebl for several useful discussions  about previous versions of this article. They also thank A. Hart for helping them correcting english and grammar errors in this article. 


\end{document}